\documentclass{amsart}
\usepackage{amssymb}
\input xy
\xyoption{all}

\newtheorem{thm}{Theorem}[section]
\newtheorem{prop}[thm]{Proposition}
\newtheorem{lem}[thm]{Lemma}
\newtheorem{cor}[thm]{Corollary}
\theoremstyle{definition}
\newtheorem{defn}[thm]{Definition}
\newtheorem{rmk}[thm]{Remark}
\newtheorem{ex}[thm]{Example}
\newtheorem*{thank}{Acknowledgments}

\newcommand{\del}[1]{\ensuremath{\delta^{#1}}}
\newcommand{\sig}[1]{\ensuremath{\sigma^{#1}}}
\newcommand{\be}[1]{\ensuremath{\sigma^{#1}}}
\newcommand{\al}[2]{\ensuremath{\delta_{#1}^{#2}}}

\newcommand{\cc}[2]{\ensuremath{c^{#1}_{#2}}}

\newcommand{\ve}{\ensuremath{\iota}}
\newcommand{\e}{\ensuremath{\epsilon}}

\newcommand{\A}{\mathcal{A}}
\newcommand{\B}{\mathcal{B}}

\newcommand{\G}{\mathbb{G}}

\newcommand{\Cat}{\mathrm{\bf Cat}}
\newcommand{\Set}{\mathrm{\bf Set}}

\newcommand{\sk}{\mathrm{sk}}
\newcommand{\cosk}{\mathrm{cosk}}
\newcommand{\op}{\mathrm{op}}
\newcommand{\dgn}{\mathrm{dgn}}

\newcommand{\id}{\mathrm{id}}

\begin{document}

\title{Levels in the toposes of simplicial sets and cubical sets}
\author{Carolyn Kennett, Emily Riehl, Michael Roy, Michael Zaks}
\date{\today}
\maketitle

\begin{abstract} The essential subtoposes of a fixed topos form a complete lattice, which gives rise to the notion of a \emph{level} in a topos. In the familiar example of simplicial sets, levels coincide with dimensions and give rise to the usual notions of $n$-skeletal and $n$-coskeletal simplicial sets. In addition to the obvious ordering, the levels provide a stricter means of comparing the complexity of objects, which is determined by the answer to the following question posed by Bill Lawvere: when does $n$-skeletal imply $k$-coskeletal?
This paper, which subsumes earlier unpublished work of some of the authors, answers this question for several toposes of interest to homotopy theory and higher category theory: simplicial sets, cubical sets, and reflexive globular sets. For the latter, $n$-skeletal implies $(n+1)$-coskeletal but for the other two examples the situation is considerably more complicated: $n$-skeletal implies $(2n-1)$-coskeletal for simplicial sets and $2n$-coskeletal for cubical sets, but nothing stronger. In a discussion of further applications, we prove that $n$-skeletal cyclic sets are necessarily $(2n+1)$-coskeletal.
\end{abstract}

\tableofcontents

\section{Introduction}

Consider a geometric morphism between toposes $\B$ and $\A$, i.e., a functor $\B \rightarrow \A$ with a finite limit preserving left adjoint. If the right adjoint is fully faithful, we say that $\B$ is a \emph{subtopos} of $\A$. If the left adjoint itself has a left adjoint, then we say $\B$ is an \emph{essential subtopos} of $\A$, in which case we have a diagram:
$$
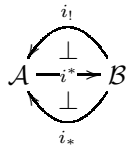
\xymatrix{ \A \ar[r]|-{i^*}^-*+{\perp}_-*+{\perp} & \B \ar@/^1.5pc/[l]^{i_*} \ar@/_1.5pc/[l]_{i_!}}$$ The right adjoint inclusion of $\B$ into $\A$ is a geometric morphism, which we think of as the sheaf inclusion of the essential subtopos. By contrast, the left adjoint inclusion, sometimes called ``essentiality,'' is not typically a geometric morphism, though in examples this is often the more natural way to think about objects of the subtopos in the context of the larger topos.

Kelly and Lawvere show that the essential subtoposes of a given topos form a complete lattice \cite{kellylawverecompletelattice}. In light of this result, each such subtopos $\B$ is referred to as a \emph{level} of $\A$. For each level $\B$, $i_! i^*$ defines a comonad $\sk_{\B}$ and $i_*i^*$ defines a monad $\cosk_{\B}$ on $\A$ such that $\sk_{\B}$ is left adjoint to $\cosk_{\B}$. 

For example, suppose $\A$ is the topos of presheaves on some small category $\Delta$. Any fully faithful inclusion $i : \Delta' \hookrightarrow \Delta$ induces functors $$\xymatrix@R=50pt{\Set^{\Delta^{\op}}\ar[r]|-{i^*}_-*+{\perp}^-*+{\perp} & \Set^{(\Delta')^{\op}} \ar@/_1.5pc/[l]_{i_!} \ar@/^1.5pc/[l]^{i_*}}$$ where $i^*$ is restriction and $i_!$ and $i_*$ are left and right Kan extension. These functors exhibit $\Set^{\Delta'^{\op}}$ as an essential subtopos of $\Set^{\Delta^{\op}}$. Up to isomorphism, the functor $i^*$ is a common retraction of $i_!$ and $i_*$, which are both fully faithful. This situation has been called \emph{unity and identity of opposites} \cite{lawvereuiop}, \cite{lawveresomethoughts}.

An object $A$ of $\A$ is $\B$-\emph{skeletal} if $A \cong \sk_{\B}A$; likewise $A$ is $\B$-\emph{coskeletal} if $A \cong \cosk_{\B}A$. A level $\B'$ is \emph{lower than} a level $\B$ if the skeletal and coskeletal inclusions of $\B'$ into $\A$ factor through the skeletal and coskeletal inclusions, respectively, of $\B$ in $\A$. In the above example, the category of presheaves on a full subcategory $\Delta'' \hookrightarrow \Delta'$ is lower than the category of presheaves on $\Delta'$ . A level $\B'$ is \emph{way below} a level $\B$ if in addition its skeletal inclusion into $\A$ factors through the coskeletal inclusion of $\B$ in $\A$, i.e., if $\B'$-skeletal implies $\B$-coskeletal. The smallest level $\B$ in the lattice of essential subtoposes of $\A$ for which this condition holds, if such a level exists, is called the \emph{Aufhebung} of $\B'$, terminology introduced by Lawvere in deference to Hegel \cite{lawveresomethoughts}. 

In three toposes which have been important for the study of homotopy theory and higher category theory --- simplicial sets \cite{gabrielzisman} \cite{maysimplicial}, cubical sets \cite{kanabstractI}, and reflexive globular sets \cite{streetpetittopos} --- levels coincide with dimensions: the category of presheaves on a small category is equivalent to the presheaves on its Cauchy completion. Up to splitting of idempotents, the distinct full subcategories of, e.g., the simplicial category $\Delta$ are the categories $\Delta_n$ on objects $[0],\ldots,[n]$ for each natural number. Thus, dimensions classify the essential subtoposes of the category of simplicial sets; a similar proof works for the other examples. For these toposes, a level $n$ is lower than a level $k$ precisely when $n \leq k$, and the question of determining the Aufhebung of the level $n$ can be stated more colloquially: when does $n$-skeletal imply $k$-coskeletal?

Naively, one might hope that $n$-skeletal implies $(n+1)$-coskeletal, and for reflexive globular sets this is indeed the case, as was first observed by Roy \cite{roytopos}. A \emph{reflexive globular set} is a presheaf on the globe category $\G$, with the natural numbers as objects and maps of the form $\sigma,\tau : n \rightarrow n+1$ such that $\tau \sigma = \sigma \sigma$ and $\tau \tau = \sigma \tau$ and  $\iota : n+1 \rightarrow n$ such that $\iota \sigma = \id = \iota\tau$. For reflexive globular sets, $n$-skeletal implies $(n+1)$-coskeletal:

\begin{ex} A reflexive globular set is $n$-skeletal if and only if the only globs above level $n$ are identities and $(n+1)$-coskeletal if and only if there exists a unique filler for each parallel pair of $k$-globs, for $k > n$. Hence, the arrows of any parallel pair of $k$-globs for $k > n$ are both equal to the identity $k$-glob on (necessarily equal) domain and codomain. Such pairs are filled uniquely by their image under $\iota$. This shows that $n$-skeletal implies $(n+1)$-coskeletal, and it is easy to construct examples to show this implication is as strong as possible.
\end{ex}

However, for simplicial sets or cubical sets, the situation is rather more complicated. Some of this work was done over 20 years ago \cite{zaksskeletal} but was never published. In light of continued interest in this problem \cite{lawverefunctorialconcepts} \cite{lawvereopenproblems}, the authors thought it was important that this work enter the literature in an easily accessible form.

The main goal of this paper is to determine the Aufhebung relation in two particular cases, that of simplicial sets and cubical sets. We will show in Theorems \ref{rcubesetthm} and \ref{ssetthm} that the Aufhebung relation for cubical sets is $2n$ and for simplicial sets is $2n-1$. The upper bound on the Aufhebung for simplicial sets is due to Zaks \cite{zaksskeletal} and the upper bound for cubical sets is due to Kennett and Roy \cite{kennettroyzaksanalysis}. The remaining author provided the examples which prove that these bounds are optimal and cleaned up the exposition.

The combinatorics involved in the proof for cubical sets is simpler, so we begin in \S \ref{cubesec} with this case, even though the proof for simplicial sets was discovered first. In \S \ref{simpsec}, we provide a complete proof for simplicial sets without reference to cubical sets, so that the reader who is only interested in that topos can skip directly there. Note that we have adopted similar notation for the face and degeneracy maps of simplicial and cubical sets to emphasize the analogy between the proofs for these toposes. As a result, notation introduced in \S \ref{cubesec} is redefined in \S \ref{simpsec}. Due to the logical independence of these sections, there should be no danger of confusion.

\begin{rmk} The apparent similarities in the arguments we present for the cubical and simplicial cases are related to the fact that the simplicial category $\Delta$ and the cube category $\mathbb{I}$ are both Reedy categories such that the degree-lowering arrows are uniquely determined by the set of their sections. (This last fact enables the proof of the Eilenberg-Zilber lemma.)  We expect that the combinatorial arguments presented in this paper could easily be adapted to similar situations, but without any other examples in mind, we were insufficiently motivated to do so ourselves.
\end{rmk}

We conclude both sections with a discussion of potential generalisations of these results that describes what seems to be possible as well as highlighting several pitfalls. In particular, we prove in \S \ref{simpsec} that the Aufhebung relation for cyclic sets, another topos of interest to homotopy theorists, is between $2n-1$ and $2n+1$. We hope that these remarks will aid future investigations relating to this problem.

\begin{thank}
The second author would like to thank the members of the Australian Category Seminar for being wonderful hosts, Ross Street for introducing her to this work, Richard Garner and Dominic Verity for stimulating conversations on this topic, and her advisor Peter May for his continued support. 
\end{thank}

\section{Aufhebung of cubical sets}\label{cubesec}

There are many variants in the notion of cubical sets \cite{grandismauricubical}, which are defined to be presheaves on various cubical categories. We present the most elementary notion, popularized by Daniel Kan \cite{kanabstractI}. Other versions of the cubical category contain the one described here, and for some of these variants, we expect that some results can be deduced from this one. See Remark \ref{othercubermk}.

We write $I$ for the poset category $0 < 1$. Note that $I$ is an interval object: there are two maps $\ve : * \rightarrow I$ with $\ve = 0,1$ from the terminal category to $I$ and a projection $I \rightarrow *$ that is a common retraction of these maps. The \emph{cube category} $\mathbb{I} \subset \Cat$ is the free monoidal category containing an interval object. 

Concretely, its objects are the elementary cubes $I^n$ for each $n \in \mathbb{N}$. Its morphisms are generated by the elementary face and degeneracy maps, defined on coordinates by 
\begin{align*} &\al{\ve}{i} : I^{n-1} \hookrightarrow I^n  &&\mbox{where} \quad \al{\ve}{i} = \langle \pi_1,...,\pi_{i-1},\ve,\pi_i,...,\pi_{n-1}\rangle, &&(i=1,...,n; \ve=0,1)\\ &\be{i} : I^n \twoheadrightarrow I^{n-1} &&\mbox{where} \quad \be{i} = \langle\pi_1,...,\pi_{i-1},\pi_{i+1},...,\pi_n\rangle,&&(i=1,..,n),\end{align*}
where $\pi_k$ denotes the $k$-th projection map for the product.  These maps satisfy the following relations 
\begin{align} 
\al{\ve}{j}\al{\upsilon}{i} &= \al{\upsilon}{i}\al{\ve}{j-1}  \quad \quad i < j \label{aleq} \\ 
\be{j}\be{i} &= \be{i}\be{j+1}\; \quad \quad i \leq j \label{beeq}\\
\be{j}\al{\ve}{i} &= \begin{cases} \al{\ve}{i}\be{j-1} & i < j \\ 
\text{id} & i = j \\
\al{\ve}{i-1}\be{j} & i > j 
 \end{cases} \label{mixeq2}
\end{align}

The category $\mathbb{I}$ has many of the good properties of the simplicial category $\Delta$.  Every morphism of $\mathbb{I}$ can be expressed uniquely as a composite $\mu\e$ of a monomorphism $\mu$ and an epimorphism $\e$. Every epimorphism $\e : I^n \rightarrow I^m$ can be factorised uniquely as  $\be{j_1}\cdots \be{j_t}$, where $1 \leq j_1 < \cdots < j_t \leq n$. These are precisely the coordinates of the domain which are deleted. Every monomorphism $\mu : I^m \rightarrow I^n$ can be factorised uniquely as $\al{\ve_1}{i_1}\cdots \al{\ve_s}{i_s}$, where $n \geq i_1 >\cdots > i_s \geq 1$. The monomorphism $\mu$ inserts the coordinate $\ve_k$ at position $i_k$.  

\begin{defn} The unique factorisation of a morphism of $\mathbb{I}$ into a product of the form $$\al{\ve_1}{i_1}\cdots\al{\ve_s}{i_s}\be{j_1}\cdots\be{j_t}$$ as described above is called the \emph{canonical factorisation} of the morphism.
\end{defn}

The cube category $\mathbb{I}$ also has a strict monoidal structure inherited from the cartesian monoidal structure on $\Cat$, which is perhaps the main advantage over $\Delta$.

A \emph{cubical set} is a functor $X: \mathbb{I}^{op}\rightarrow \Set$.  We will write $X_{n}$ for the image of the object $I^{n}$ under the functor $X$ and call elements of this set $n$-\emph{cubes}.  Each arrow $\tau:I^{m} \rightarrow I^{n}$ in $\mathbb{I}$ gives rise to a function $X_{n} \rightarrow X_{m}$ whose value at $x \in X_{n}$ is denoted $x\tau$. An $n$-cube $x$ is \emph{degenerate} if there exists an epimorphism $\e : I^n \rightarrow I^m$ with $m < n$ and an $m$-cube $y$ such that $x = y\e$.

Write $\mathbb{I}_n$ for the full subcategory of $\mathbb{I}$ on the objects $I^1,\ldots, I^n$. The essential subtopos $\Set^{\mathbb{I}_n^{\op}}$ of the category $\Set^{\mathbb{I}^{\op}}$ of cubical sets induces a pair of adjoint functors $\sk_n \dashv \cosk_n$ on $\Set^{\mathbb{I}^{\op}}$. Concretely, the $n$-\emph{skeleton} $\sk_{n}X$ of a cubical set $X$ consists of those cubes $x \in X_{m}$ such that there exist $y \in X_k$ with $k \leq n$ and an epimorphism $ \epsilon : I^m \rightarrow I^k$ in $\mathbb{I}$ such that $x=y \epsilon$. As in the introduction, a cubical set $X$ is $n$-\emph{skeletal} if it is isomorphic to its $n$-skeleton, i.e., when each $m$-simplex with $m >n$ is degenerate.

\begin{defn} An $k$-\emph{sphere} or $k$-\emph{cycle} $c$ in $X$ is a sequence of $(k-1)$-cubes $\cc{0}{1}, \cc{1}{1}, \ldots ,\cc{0}{k},\cc{1}{k}$ satisfying the \emph{cycle equations} \begin{equation}\label{ccycleeq} \cc{\ve}{j}\al{\upsilon}{i} = \cc{\upsilon}{i}\al{\ve}{j-1} \hspace{.5cm} \text{for}\ i <j. \end{equation} 
\end{defn}

\begin{ex} Let $\cc{\ve}{i} = x \al{\ve}{i}$ for some $k$-cube $x$ in a cubical set $X$. Then $c$ is a $k$-sphere in $X$. 
\end{ex}

As in the introduction, $X$ is $n$-\emph{coskeletal} if it is isomorphic to $\cosk_n X$. Concretely, this says that for any $k$-sphere in $X$ with $k >n$ there is a unique $k$-cube $y$ such that $y \al{\ve}{i} = \cc{\ve}{i}$ for all 
$0 \leq i \leq k$ and $\ve=0,1$.  

Importantly, we have an Eilenberg-Zilber type lemma for cubes.

\begin{lem}\label{cezlem}
For each $x \in X_n$, there is a unique non-degenerate $y \in X_k$ for some $k \leq n$ together with a unique epimorphism $\e : I^n \rightarrow I^k$ such that $x = y\e$.
\end{lem}
\begin{proof} 
Existence is obvious. For uniqueness, suppose $x=y\e$ and $x=y'\e'$ satisfy these conditions, where $y \in X_k$ and $y' \in X_{k'}$. Let $\mu$ and $\mu'$ be sections for $\e$ and $\e'$ respectively.  Then
\[y= y\e\mu = x \mu = y' \e' \mu \]
Since $y$ is non-degenerate, the epimorphism portion of the canonical factorisation of $ \e' \mu$ must be trivial; thus $ \e' \mu :I^k \rightarrow I^{k'}$ is a monomorphism. By a similar argument for $\mu'$ and $\e$ we have a monomorphism $\e\mu' : I^{k'} \rightarrow I^k$. So $k=k'$, which means that these monomorphisms are both identities, and hence that $y=y'$. It follows that $\mu$ is a section for both $\e$ and $\e'$. In the cube category $\mathbb{I}$, a section uniquely determines its retraction; hence $\e=\e'$.
\end{proof}

When $x = y\e$ as in the lemma, we say that $x$ \emph{has degeneracy} $n-k$ and write $\dgn(x)=n-k$. Note, the canonical factorisation of $\e$ will have the form $\be{i_1}\cdots\be{i_{n-k}}$.

\begin{lem}\label{cubedgnlem}
Let $x$ be an $n$-cube. Then for $\ve = 0,1$ and all appropriate $i$
\begin{align*} \dgn(x\be{i}) & =  \dgn(x) +1 \\ 
\dgn(x\al{\ve}{i}) & \geq \dgn(x) -1. 
\end{align*}
\end{lem}
\begin{proof}
Obvious using Lemma \ref{cezlem}.
\end{proof}

The degenerate cube $x\be{i}$ has $x$ for its 0-th and 1-st faces, perpendicular to the $i$-th coordinate direction. All other faces are degenerate, even if $x$ is non-degenerate. We are interested in identifying which faces of a degenerate cube are least degenerate. Hence, the following definition.

\begin{defn}
Say that $1 \leq i \leq n$ \emph{reduces} an $n$-cube $x$ when $\dgn(x\al{\ve}{i}) = \dgn(x)-1$ for some $\ve$. 
\end{defn}

\begin{rmk}
Note if $x$ is reduced by $i$ then $$x\al{0}{i} = (x\al{1}{i}\be{i})\al{0}{i} = x\al{1}{i}(\be{i}\al{0}{i}) = x\al{1}{i}.$$
\end{rmk}

There are several equivalent characterisations of this condition, as indicated by the following lemma, whose proof is an easy exercise.

\begin{lem}\label{ctfaelem} The following are equivalent: \\ \indent \emph{(i)} $i$ reduces $x$. \\ \indent \emph{(ii)} $\dgn(x\al{0}{i}) = \dgn(x\al{1}{i})=\dgn(x)-1$. \\ \indent \emph{(iii)} the epimorphism of the Eilenberg-Zilber decomposition of $x$ deletes the $i$-th \\ \indent \qquad coordinate.\\ \indent \emph{(iv)} $\be{i}$ appears in the canonical factorisation of the epimorphism in the Eilenberg-\indent \qquad Zilber decomposition of $x$. \\ \indent \emph{(v)}  $x = x \al{\ve}{i}\be{i}$ for some $\ve$. \\ \indent \emph{(vi)} $x = x\al{0}{i}\be{i} = x\al{1}{i}\be{i}$.
\end{lem}

Note the following obvious but useful consequence of these equivalent characterisations.

\begin{lem}\label{tricklem}
Suppose $x$ and $y$ are $n$-cubes which are both reduced by $i$. If $x\al{\ve}{i}=y\al{\upsilon}{i}$ for some $\ve$ and $\upsilon$ then $x=y$.
\end{lem}

The main technical tool in the computation of the Aufhebung relation for cubical sets is the following proposition, which we will use to show that spheres consisting of highly degenerate cubes can be filled by a cleverly chosen degenerate copy of one of the least degenerate constituent faces.

\begin{prop}\label{techcubeprop}
Let $X$ be a cubical set which is $n$-skeletal and let $c$ be a $k$-sphere with faces $\cc{0}{1},\cc{1}{1},\ldots, \cc{0}{k},\cc{1}{k}$, all degenerate. Let $r$ be the minimal degeneracy of the faces $\cc{\ve}{i}$ and let $m$ be the smallest ordinal with $\dgn(\cc{\ve}{m})=r$ for some $\ve$. If $k<2r+2$ then this sphere is filled by $\cc{\ve}{m}\be{m}$.
\end{prop}
\begin{proof}
Suppose $\dgn(\cc{\ve}{m})=r$ with $m$ minimal. We will make repeated use of the set $M = \{ j_1,\ldots, j_r\}$ of ordinals that reduce $\cc{\ve}{m}$; i.e., write $M$ for the set of ordinals which appear in the canonical factorisation $\be{j_1}\cdots\be{j_r}$ of the epimorphism part of the Eilenberg-Zilber decomposition of $\cc{\ve}{m}$.

By a dimension argument, $M \subset \{1,\ldots, k-1\}$. In fact, because we chose $m$ to be minimal, each $j \in M$ is such that $j \geq m$: if $j<m$ reduces $\cc{\ve}{m}$ then by the cycle equations,
\[
r-1=\dgn(\cc{\ve}{m}\al{\upsilon}{j})=\dgn(\cc{\upsilon}{j}\al{\ve}{m})\] which means $\dgn(\cc{\upsilon}{j})=r$, contradicting our choice of $m$.

First, we show that $\cc{0}{m} = \cc{1}{m}$. For any $j \in M$ and some fixed $\upsilon$, \begin{equation}\label{jinMeq}r-1 = \dgn ( \cc{\ve}{m} \al{\upsilon}{j} ) =\dgn( \cc{\upsilon}{j+1} \al{\ve}{m})\end{equation} by the cycle equations. Because $r$ is the minimal degeneracy of the faces of $c$, this says that $m$ reduces $\cc{\upsilon}{j+1}$ which means that $\cc{\upsilon}{j+1} \al{0}{m} = \cc{\upsilon}{j+1}\al{1}{m}$ by Lemma \ref{ctfaelem}. By applying the cycle equations to both sides of this equality, we see that the $\al{\upsilon}{j}$ faces of $\cc{0}{m}$ and $\cc{1}{m}$ are equal and Lemma \ref{tricklem} implies that $\cc{0}{m} = \cc{1}{m}$. 

Henceforth, we write $c_m$ for $\cc{0}{m}=\cc{1}{m}$. We wish to show that $\cc{\ve}{u} = c_m\be{m}\al{\ve}{u}$ for all faces of $c$. Immediately from (\ref{mixeq2}), $\cc{\ve}{m} = c_m\be{m}\al{\ve}{m}$ for $\ve=0,1$. We complete the proof by dividing the remaining faces into three cases.

Part I: ($\cc{\ve}{j+1} = c_m\be{m}\al{\ve}{j+1}$ for $j \in M$). By (\ref{jinMeq}), $m$ reduces $\cc{\ve}{j+1}$ when $j \in M$.  Hence, 
\[ \cc{\ve}{j+1} = \cc{\ve}{j+1}\al{\ve}{m}\be{m} = c_m\al{\ve}{j} \be{m} = c_m\be{m} \al{\ve}{j+1}\] by (\ref{ccycleeq}) and then (\ref{mixeq2}), recalling that $j\geq m$. This is what we wanted to show.

Part II: ($\cc{\ve}{u} = c_m\be{m}\al{\ve}{u}$ for all $u<m$). If $u<m$ then $\cc{\ve}{u}$ must be reduced by at least $r+1$ ordinals in $\{1,\ldots,k-1\}$, by minimality of $m$.  If $k-1 < 2r+2$ then at least one of these lies in the $r+1$ element set $ \{m-1 \} \cup M$.  Call this
element $p$.  Then we have
\begin{align*}
\cc{\ve}{u} & =  \cc{\ve}{u}\al{\upsilon}{p}\be{p} && \mbox{$p$ reduces $ \cc{\ve}{u}$} \\
& =  \cc{\upsilon}{p+1}\al{\ve}{u}\be{p} && \mbox{cycle equation ($u<p+1$)}  \\
& =  \cc{\upsilon}{p+1}\be{p+1}\al{\ve}{u}&& \mbox{cubical identity.} \\ \intertext{If $p=m-1$ this is exactly what we want. Otherwise, $p \in M$ and $\cc{\upsilon}{p+1} =c_m\be{m}\al{\upsilon}{p+1}$ by Part I. By substitution}
\cc{\ve}{u} 
 & =  c_m\be{m}\al{\upsilon}{p+1}\be{p+1}\al{\ve}{u} \\
 & =  c_m\al{\upsilon}{p}\be{m}\be{p+1}\al{\ve}{u} && \mbox{cubical identity ($m<p+1$)} \\
 & =  c_m\al{\upsilon}{p}\be{p}\be{m}\al{\ve}{u} && 
\mbox{cubical identity} \\
 & =  c_m\be{m}\al{\ve}{u}  && \mbox{$p$ reduces $c_m$.} 
\end{align*}
This is what we wanted to show.

Part III: ($\cc{\ve}{u} = c_m\be{m}\al{\ve}{u}$ for $u>m$ and $u-1 \notin M$). Let $$K = \{ m\} \cup \{ j+1 \mid j \in M, j+1< u \} \cup \{ j \mid j \in M, j+1 > u \}.$$ Because $u-1 \notin M$ and $u>m$, $K$ has $r+1$ elements. Tautologically, $K\subseteq\{1,2,...,k-1\}$. Because $\cc{\ve}{u}$ is reduced by at least $r$ elements, if $k-1<2r+1$, there is a $p\in K$ that reduces $\cc{\ve}{u}$.

Case 1: ($p<u$ so either $p=m$ or $p-1 \in M$).
If $p=m$ then \begin{align*} \cc{\ve}{u} &= \cc{\ve}{u}\al{\ve}{m}\be{m} && \mbox{$m$ reduces $\cc{\ve}{u}$} \\
           & = \cc{\ve}{m}\al{\ve}{u-1}\be{m} &&
\mbox{cycle equation ($m<u$)}  \\
           & = \cc{\ve}{m}\be{m}\al{\ve}{u} &&
\mbox{cubical identity}\\ \intertext{as desired. If $p-1 \in M$, then}
\cc{\ve}{u} &= \cc{\ve}{u} \al{\ve}{p}\be{p} && 
\mbox{$p$ reduces $\cc{\ve}{u}$} \\
 & =  \cc{\ve}{p}\al{\ve}{u-1}\be{p} &&
\mbox{cycle equation ($p<u$)}  \\
 & =  \cc{\ve}{p}\be{p}\al{\ve}{u} &&
\mbox{cubical identity} \\
             & = c_{m}\be{m}\al{\ve}{p}
\be{p}\al{\ve}{u} && \mbox{Part I}\\
 & = c_m\al{\ve}{p-1}\be{p-1}
\be{m}\al{\ve}{u} && \mbox{(\ref{mixeq2}) then (\ref{beeq}) ($m <p$)} \\
             & = c_m\be{m}\al{\ve}{u} &&
\mbox{${p-1}$ reduces $c_m$}
\end{align*}
which is the desired conclusion.

Case 2: ($p\geq u$ and hence $p \in M$).
\begin{align*}
\cc{\ve}{u} & =  \cc{\ve}{u}\al{\ve}{p}\be{p} &&
\mbox{$p$ reduces $\cc{\ve}{u}$} \\
		 & =  \cc{\ve}{p+1}\al{\ve}{u}\be{p} && \mbox{cycle equation $(u \leq p)$} \\
                 & =  \cc{\ve}{p+1}\be{p+1}\al{\ve}{u} && 
\mbox{cubical identity} \\
             & =  c_m\be{m}\al{\ve}{p+1}
\be{p+1}\al{\ve}{u} && \mbox{Part I}	 \\
        	& = c_m\al{\ve}{p}\be{p}\be{m}\al{\ve}{u} && \mbox{(\ref{mixeq2}) then (\ref{beeq}) ($m \leq p$)}
 \\    & =  c_m\be{m}\al{\ve}{u} && \mbox{$p$ reduces $c_m$} 
\end{align*}

Combining these cases, we have shown that if $k < 2r+2$ then $\cc{\ve}{u} = c_m\be{m}\al{\ve}{u}$ for all $u=1,\ldots,k$ and $\ve=0,1$. Hence, $c_m\be{m}$ is a filler for the $k$-sphere $c$.
\end{proof}

In order to use Proposition \ref{techcubeprop} to prove that $n$-skeletal implies $2n$-coskeletal, we must also show that the filler it constructs for high dimensional spheres is unique. This follows from the following lemma, which states that degenerate cubes are uniquely determined by their boundaries.

\begin{lem}\label{cubeuniquelem}
If $x$ and $y$ are two degenerate $k$-cubes in $X$ with the same faces, i.e., if 
 $x\al{\ve}{i} = y\al{\ve}{i}$ for all $1\leq i\leq k$ and 
$\ve=0,1$,  then $x=y$.
\end{lem}
\begin{proof} Because both $x$ and $y$ are degenerate there is some $i$ that reduces $x$ and some $j$ that reduces $y$. If $i=j$ we are done by Lemma \ref{tricklem}, so we suppose without loss of generality that $i <j$. Then 
\[ x = x\al{\ve}{i}\be{i} = y\al{\ve}{i}\be{i} = y\al{\ve}{j}\be{j}\al{\ve}{i}\be{i} = y \al{\ve}{i}\be{i}\al{\ve}{j}\be{j} = x\al{\ve}{i}\be{i}\al{\ve}{j}\be{j} = x\al{\ve}{j}\be{j}\] using the hypothesis that $x$ and $y$ share the same faces, the definition of $i$ and $j$, and the cubical identities. This says that $j$ reduces $x$ as well as $y$ and the conclusion follows by Lemma \ref{tricklem}.
\end{proof}

It is easy to check that a $0$-skeletal cubical set is $1$-coskeletal. For larger $n$, we use the preceding work to prove our main result.

\begin{thm}\label{rcubeupperboundthm}
If a cubical set is $n$-skeletal, it is $2n$-coskeletal. Hence,
the Aufhebung relation for the topos of cubical sets is bounded above by $2n$. 
\end{thm}
\begin{proof} We must show that any $k$-sphere in an $n$-skeletal cubical set $X$ with $k> 2n$ can be filled uniquely. The inequality $k > 2n$ can be rewritten as $k < 2(k-1-n)+2$. The faces of a $k$-sphere are $(k-1)$-cubes, which therefore have degeneracy at least $k-1-n$.  Applying Proposition \ref{techcubeprop}, any $k$-sphere has a filler. By Lemma \ref{cubeuniquelem}, it's unique.
\end{proof}
 
\begin{ex} Let $X$ be the $n$-skeletal cubical set generated by a single vertex $v$ and two $n$-cubes $x$ and $y$, with each face equal to the $(n-1)$-cube $v\be{1}\cdots\be{n-1}$. We define a $2n$-sphere with faces $$\cc{\ve}{i} = \begin{cases} x\be{1}\cdots\be{n-1} & 1 \leq i \leq n \\ y\be{n+1}\cdots\be{2n-1} & n < i \leq 2n. \end{cases}$$ No cube of $X$ contains both $x$ and $y$ as faces, so this sphere has no filler.\end{ex}

\begin{thm}\label{rcubesetthm}
The Aufhebung relation for the topos of cubical sets is $2n$.
\end{thm}
\begin{proof} Immediate from Theorem \ref{rcubeupperboundthm} and the preceding example, which shows that an $n$-skeletal cubical set is not necessarily $(2n-1)$-coskeletal.
\end{proof}

\begin{rmk}\label{othercubermk}
In the literature, there are a plethora of definitions of a cubical category $\mathbb{C}$: e.g., cubical categories with partial diagonals, conjunctions, connections, interchange, etc. These typically contain $\mathbb{I}$ as a non-full subcategory. The categories $\mathbb{C}$ are usually not Reedy categories, but are often generalized Reedy categories, in the sense of Berger and Moerdijk \cite{bergermoerdijkextension}. For such categories, one may again describe canonical factorisations, which are typically only unique up to isomorphism.

For any of these examples, levels again coincide with dimensions. When the canonical factorisations in $\mathbb{C}$ are particularly nice, restriction along the inclusion $\mathbb{I} \rightarrow \mathbb{C}$ will be compatible with the skeletal and coskeletal inclusions of the levels, though this is by no means always the case. The example of cubical categories with partial diagonals has this nice property, and a straightforward argument due to Kennett and Roy \cite{kennettroyzaksanalysis} can be used to prove that the Aufhebung relation is again $2n$. 

More frequently, the restrictions are compatible with only one of the level inclusions. In particular, whenever some epimorphisms in $\mathbb{C}$ cannot be factored as an epimorphism in $\mathbb{I}$ followed by something else, $n$-skeletal presheaves on $\mathbb{C}$ will not be $n$-skeletal as presheaves on $\mathbb{I}$. Nonetheless, the above results at least provide a bound for the Aufhebung relation. This sort of situation is discussed in Corollary \ref{cycliccor}.
\end{rmk}

\section{Aufhebung of simplicial sets}\label{simpsec}

Let $\Delta$ be the category of finite non-empty ordinals and order preserving maps. The objects of $ \Delta$ are the ordered sets $\{0,1,2,\ldots,n\}$ denoted by $[n]$ for each non-negative integer $n$. The morphisms of $\Delta$ are order preserving maps. These are generated by the elementary face and degeneracy maps: for each $n$ there are $n+1$ monics, 
\[
\del{0} \geq \del{1} \geq \ldots \geq \del{n}:[n-1] \rightarrow [n]
\]
such that the image of $\del{i}$ does not contain $i$ and there are $n$ epics 
\[ \sig{0} \leq \sig{1} \leq \ldots \leq \sig{n-1}:[n] \rightarrow [n-1]
\]
such that two elements map to $i \in [n]$. Explicitly, 
\begin{eqnarray*} 
\sig{i}(j)=\left \{ \begin{array}{lc}
                      j & j \leq i \\                      j-1 & j>i \\                     
\end{array} 
\right. & \quad \del{i} (j)=\left \{ \begin{array}{lc}
                     j & j<i \\                      j+1 & j \geq i. \\
\end{array} \right. \\     
\end{eqnarray*}
These maps satisfy the following relations
\begin{eqnarray}
\del{j} \del{i} & = & \del{i}\del{j -1} \,\,\, \qquad 
i < j \label{deleq} \\
\sig{j} \sig{i} & = & \sig{i} \sig{j +1} \,\, \qquad 
i \leq j \label{sigeq} \\
\sig{j} \del{i} & = & \left \{ \begin{array}{ll}
		\del{i} \sig{j -1} & i < j \\
		\mbox{id} & i = j \,\, \mbox{or} \,\, i = j +1 \\
		\del{i -1} \sig{j} & i > j +1\label{mixeq} \\ 
\end{array} \right.
\end{eqnarray}

Any arrow of $\Delta $ can be written uniquely as a composite
$\mu\epsilon $ of a monic $\mu $ and an epic $\epsilon$. Each monic 
$\mu:[m] \rightarrow [n]$ can be factorised uniquely as $\mu = \del{i_{1}}\del{i_{2}} \ldots \del{i_{s}}$ where $n \geq i_{1}>i_{2}> \cdots >i_{s}\geq0$ are the elements of $[n]$ which are not in the image of $\mu$.  Each epic $\epsilon:[n] \rightarrow [m]$ is uniquely of the form $\epsilon = \sig{j_{1}} \cdots \sig{j_{t}}$ where $0\leq j_{1}<j_{2}< \cdots <j_{t} \leq n-1$ are the elements $j \in [n]$ with $\epsilon(j)=\epsilon(j+1)$. 

\begin{defn} The unique factorisation of a morphism of $\Delta$ into a product of the form $$\del{i_s}\cdots \del{i_1}\sig{j_1}\cdots \sig{j_t}$$ as described above is called the \emph{canonical factorisation} of the morphism.
\end{defn} 

A simplicial set is a functor $X:\Delta^{\!\op} \rightarrow \Set$.  We will write $X_{n}$ for the image of the object $[n]$ under the functor $X$.  Elements of $X_n$ are called $n$-\emph{simplices}. Each arrow $\tau:[m] \rightarrow [n]$ in $\Delta$ gives rise to a function $X_{n} \rightarrow X_{m}$ in $\Set$ whose value at $x \in X_{n}$ is denoted $x\tau$. Alternatively, a simplicial set $X$ consists of sets $X_n$ for each $n$ together with right actions by the $\del{i}$, which take $n$-simplices to $(n-1)$-simplices, and the $\sig{j}$, which take $n$-simplices to $(n+1)$-simplices. An $n$-simplex $x$ is \emph{degenerate} if there exists epimorphism $\epsilon:[n] \rightarrow [m]$ 
with $m<n$ and an $m$-simplex $y$ such that  $x=y\epsilon$.

Write $\Delta_n$ for the full subcategory of $\Delta$ on the objects $[0],\ldots,[n]$. The essential subtopos $\Set^{\Delta_n^{\op}}$ of the category $\Set^{\Delta^{\op}}$ of simplicial sets induces pair of adjoint functors  $\sk_n \dashv \cosk_n$ on $\Set^{\Delta^{\op}}$. Concretely, the $n$-\emph{skeleton} $\sk_{n}X$ of a simplicial set $X$ is the subcomplex of $X$ that is formed by all $x \in X_{k}$ such that there exist $y \in X_m$ with $m \leq n$ and an epimorphism $\epsilon :[k] \rightarrow [m]$ in $\Delta$ such that $x=y \epsilon$. As above, a simplicial set $X$ is $n$-\emph{skeletal} if it is isomorphic to its $n$-skeleton, i.e., when each $k$-simplex with $k >n$ is degenerate.

\begin{defn} An $k$-\emph{sphere} or $k$-\emph{cycle} $c$ in $X$ is a sequence of $(k-1)$-simplices $c_{0}, \ldots ,c_{k}$ satisfying the \emph{cycle equations} \begin{equation}\label{cycleeq} c_{j}\del{i} = c_{i}\del{j-1} \hspace{.5cm} \text{for}\ i <j. \end{equation} 
\end{defn}

\begin{ex} Let $c_i = x \del{i}$ for some $k$-simplex $x$ in a simplicial set $X$. Then $c$ is a $k$-sphere in $X$. 
\end{ex}

As in the introduction, $X$ is $n$-\emph{coskeletal} if it is isomorphic to $\cosk_n X$. Concretely, this says that for any $k$-sphere in $X$ with $k >n$ there is a unique $k$-simplex $y$ such that $y \del{i} = c_{i}$ for all 
$0 \leq i \leq k$.  

The following lemma is very important.

\begin{lem}[Eilenberg-Zilber lemma]\label{ezlem}
For each $x \in X_{n}$, there is a unique non-degenerate $y \in X_k$ for some $k\leq n$ together with a unique epimorphism $\e : [n] \rightarrow [k]$ such that $x = y\e$.
is unique.
\end{lem}
\begin{proof} Similar to \ref{cezlem} or see \cite[pp 26-27]{gabrielzisman}.
\end{proof}

When $x = y\e$ as in the lemma, we say $x$ \emph{has degeneracy} $n-k$ and write $\dgn(x)=n-k$. Note, the canonical factorisation of $\e$ will have the form $\sig{j_1}\cdots\sig{j_{n-k}}$.

\begin{lem}\label{dgnlem}
Let $x$ be an $n$-simplex. Then for $0\leq i \leq n$
\begin{align*} \dgn(x\sig{i}) & =  \dgn(x) +1 \\
\dgn(x\del{i}) & \geq \dgn(x) -1.
\end{align*}
\end{lem}
\begin{proof}
Obvious using Lemma \ref{ezlem}.
\end{proof}

Using (\ref{mixeq}), the degenerate simplex $x\sig{i}$ has $x$ as its $i$-th and $(i+1)$-th faces and degeneracies for all of the other faces, even if $x$ is non-degenerate. We will be interesting in identifying which faces of a degenerate simplex are least degenerate. Hence, the following definition. 

\begin{defn}
Say $i \in [n]$ {\em reduces} an $n$-simplex $x$ when $\dgn(x\del{i})=\dgn(x)-1$. 
\end{defn}

\begin{lem}\label{tfaelem} Let $x$ be an $n$-simplex and suppose $0 \leq i \leq n$. The following are equivalent: \\ \indent \emph{(i)} $i$ reduces $x$ \\ \indent \emph{(ii)} $x=x\del{i}\sig{i}$ or $x=x\del{i}\sig{i-1}$ 
\end{lem}
\begin{proof}
(i) $ \Rightarrow $ (ii). Write $x = y\epsilon$ as in Lemma \ref{ezlem}. If $\epsilon\del{i}$ is not epic, it factors through $[k-1]$, which contradicts the fact that $\dgn(x\del{i})=n-k-1$. Hence, $\epsilon\del{i}$ is epic, which means that $\e (i) = \e(i-1)$ or $\e(i) = \e(i+1)$. In the first case, $\e = \e \del{i}\sig{i}$ and in the second $\e = \e\del{i}\sig{i-1}$, which implies (ii).

(ii) $ \Rightarrow $ (i). Let $j=i-1$ or $j=i$, as appropriate. By (ii) and Lemma \ref{dgnlem}, $$\dgn(x)=\dgn(x\del{i}\sig{j})=\dgn(x\del{i})+1 \geq \dgn(x).$$   So the inequality is an equality and $i$ reduces $x$.
\end{proof}

By the lemma, if $i$ reduces $x$ then $x$ is a degenerate copy of its $i$-th face. Either, this $i$-th face appears as the $i$-th and $(i+1)$-st faces of $x$, in which case $x = x \del{i}\sig{i}$; or it's the $(i-1)$-st and $i$-th faces, in which case $x=x\del{i}\sig{i-1}$. To enable subsequent accounting, we artificially choose to prefer the former and introduce terminology to distinguish these situations.

\begin{defn}
Say $i\in [n]$ \emph{properly reduces} $x$ when $x=x\del{i}\sig{i}$.
\end{defn}

\begin{ex}\label{prex} If $x = y\sig{i}$, then \[ x = y \sig{i} = y (\sig{i} \del{i}) \sig{i} = y \sig{i} (\del{i} \sig{i}) = x \del{i} \sig{i} \] by (\ref{mixeq}). So $i$ properly reduces $x$.
\end{ex}

\begin{rmk}\label{prrmk} It follows from the computation in \ref{prex} that $i$ properly reduces $x$ if and only if $\sig{i}$ appears in the canonical factorisation of the epimorphism of the Eilenberg-Zilber decomposition of $x$. In particular, $x$ is properly reduced by exactly $\dgn(x)$ ordinals.
\end{rmk}

\begin{lem}\label{oprlem} If $i$ properly reduces $x$, then $i+1$ reduces $x$ but not necessarily properly.
\end{lem}
\begin{proof} Assuming $i$ properly reduces $x$, then \[ x \del{i+1} \sig{i} =  (x \del{i} \sig{i}) \del{i+1} \sig{i} = x \del{i} (\sig{i} \del{i+1}) \sig{i} = x \del{i} \sig{i} = x,\] which says that $i+1$ reduces $x$ by Lemma \ref{tfaelem}.
\end{proof}

\begin{lem}\label{prlem} If $i$ reduces $x$ but $i$ does \emph{not} properly reduce $x$, then $i-1$ properly reduces $x$.
\end{lem}
\begin{proof} By Lemma \ref{tfaelem}, if $i$ reduces $x$ but not properly, then $x = x\del{i}\sig{i-1}$. By substitution and (\ref{mixeq}), \[ x \del{i-1} \sig{i-1} = x \del{i} \sig{i-1} \del{i-1} \sig{i-1} = x \del{i} \sig{i-1} = x, \] which says that $i-1$ properly reduces $x$.
\end{proof}

It will be clear from the following lemma that the converse to Lemma \ref{prlem} does not hold.

\begin{lem}Suppose $x=y\e$ with $y$ non-degenerate and $\e$ epic.  
Then $i$ properly reduces $x$ precisely when $\e(i) = \e(i+1)$.
\end{lem}
\begin{proof}  If $\e(i)= \e(i+1)$ then $\e\del{i}\sig{i}=\e$, so $x \del{i}\sig{i} =x$, which says that $i$ properly reduces $x$. Conversely, if $y\e\del{i}\sig{i} = y\e$, we saw in the proof of Lemma \ref{tfaelem} that $\e\del{i}$ is epi, so by uniqueness $\e = \e\del{i}\sig{i}$ and hence $\e(i) =\e(i+1)$.
\end{proof}

The main technical tool in the computation of the Aufhebung relation for simplicial sets is the following proposition, which we will use to show that spheres consisting of highly degenerate simplices can be filled by a cleverly chosen degenerate copy of one of the least degenerate constituent faces.

\begin{prop}\label{techprop} Let $X$ be a simplicial set which is $n$-skeletal, and let $c$ be a $k$-sphere in $X$ whose faces $c_0,\ldots, c_k$ all have degeneracy at least 2. Let $r$ be the minimal degeneracy the faces $c_i$  
and let $m$ be the smallest ordinal with $\dgn(c_{m})=r$.  If $k<2r+3$ then 
this sphere is filled by $c_m\sig{m}$.
\end{prop}

To aid the proof of this proposition, we use some technical lemmas.

\begin{lem}\label{techlem1} Let $c$ be a $k$-sphere in $X$ with all faces degenerate. Let $r$ be the minimal degeneracy of the faces $c_0,\ldots, c_k$, and let $m$ be the first ordinal with $\dgn(c_m)=r$. Suppose $j$ reduces $c_m$. Then \\ \indent \emph{(a)} $j\geq m$ and $m$ properly reduces $c_{j+1}$. \\ \indent \emph{(b)} Furthermore, $c_{j+1} = c_m \sig{m} \del{j+1}$. \\ \indent \emph{(c)} If $m$ reduces $c_m$, then $m$ properly reduces $c_m$ and $c_m = c_{m+1}$.
\end{lem}
\begin{proof} (a). If $j$ reduces $c_m$ and $j < m$, then $\dgn(c_j \del{m-1} ) = \dgn(c_m \del{j}) = r-1$. By minimality of $r$, $\dgn(c_j) =r$, contradicting minimality of $m$. Hence $j \geq m$. By the cycle equations,  $\dgn(c_{j+1}\del{m}) = \dgn( c_m\del{j})= r-1$; but $\dgn(c_{j+1}) \geq r$, which means that $m$ reduces $c_{j+1}$. If $m$ does not properly reduce $c_{j+1}$ then $m-1$ does, in which case $\dgn(c_{m-1} \del{j}) = \dgn( c_{j+1} \del{m-1}) = r-1$, contradicting minimality of $m$. So $m$ properly reduces $c_{j+1}$. 

(c). If $m$ reduces $c_m$, then $m$ properly reduces $c_m$ because $m-1$ cannot. Taking $j=m$ in part (a), we see that $m$ properly reduces $c_{m+1}$. By the cycle equations and the fact that $m$ properly reduces $c_m$ and $c_{m+1}$, we deduce that \[ c_m = c_m \del{m} \sig{m} = c_{m+1} \del{m} \sig{m} = c_{m+1}.\]

(b). By part (a), $m$ properly reduces $c_{j+1}$. If $j>m$, \[ c_{j+1} = c_{j+1} \del{m}\sig{m} = c_m \del{j} \sig{m} = c_m \sig{m} \del{j+1}\] by (\ref{cycleeq}) then (\ref{mixeq}). If $j=m$, $\sig{m}\del{m+1}=\text{id}$ and the conclusion follows immediately from (c).
\end{proof}

\begin{lem}\label{techlem2} Let $c$, $X$, $r$, $m$ be as above. Then there are least $r+2$ faces $c_u$ of $c$ such that $\dgn(c_u)=r$ and $c_u = c_m \sig{m}\del{u}$.
\end{lem}
\begin{proof} Let $\e$ be the unique epimorphism the Eilenberg-Zilber decomposition of $c_m$, and let $$M = \{ j \mid \e(j) = \e(j+1)\}$$ be the set of indices $m \leq j \leq k-2$ that properly reduce $c_m$; the lower bound is from Lemma \ref{techlem1} and the upper bound is by a dimension argument. By Remark \ref{prrmk}, $|M|=r$. By Lemma \ref{oprlem}, if $l \in [k-1]$ is the smallest ordinal not in $M$, then $l$ reduces $c_m$ also.

Let $j \in M \cup \{l\}$. Then $$\dgn(c_{j+1}\del{m}) = \dgn(c_m\del{j}) = r-1,$$ which implies that $\dgn(c_{j+1}) = r$. This gives us a set of $r+2$ elements of degeneracy $r$ \[\{c_m\} \cup \{ c_{j+1} \mid j \in M\} \cup \{c_{l+1}\}.\] Each of these faces also satisfies $c_u = c_m \sig{m} \del{u}$; the first one trivially by (\ref{mixeq}) and the remaining $r+1$ by Lemma \ref{techlem1} (b). 
\end{proof}

In Proposition \ref{techprop}, we will show that a sufficiently degenerate sphere $c$ is filled by $c_m\sig{m}$ where $m$ is the smallest ordinal corresponding to a face of minimal degeneracy. In order for $c_m\sig{m}$ to fill the sphere $c$, we must have $$c_{m+1} = c_m\sig{m}\del{m+1}=c_m.$$ The hardest part of the proof will be verifying this condition, so we tackle this first.

\begin{lem}\label{mplus1lem} Let $X$ be a simplicial set which is $n$-skeletal, and let $c$ be a $k$-sphere in $X$ whose faces $c_0,\ldots,c_k$ all have degeneracy at least $2$. Let $r$ be the minimal degeneracy of the faces $c_i$, and let $m$ be the smallest ordinal with $\dgn(c_m)=r$. If any of the following hold, then $c_m=c_{m+1}$.\\ \indent \emph{(a)} $m$ reduces $c_m$. \\ \indent \emph{(b)} some $j$ properly reduces both $c_m$ and $c_{m+1}$. \\ 
\indent \emph{(c)} $m$ reduces $c_{m+1}$ and $\dgn(c_{m+1})=r$. \\ 
\indent \emph{(d)} $k < 2r+3$.\\
In particular, if $k < 2r+3$, then $c_m = c_{m+1}$.
\end{lem}
\begin{proof}
(a). This is shown in Lemma \ref{techlem1} (c).

(b). By Lemma \ref{techlem1} (a), $j \geq m$ and by part (a) just completed, it suffices to assume that $j > m$. Then \begin{align*} c_{m+1} &= c_{m+1}\del{j}\sig{j} & & \mbox{$j$ properly reduces $c_{m+1}$} \\ &= c_{j+1} \del{m+1} \sig{j} & & \mbox{cycle equations $(m<j)$} \\ &= c_m\sig{m}\del{j+1}\del{m+1}\sig{j} & & \mbox{Lemma \ref{techlem1} (b)} \\&= c_m\del{j}\sig{j}\sig{m}\del{m+1} & & \mbox{(\ref{mixeq}) then (\ref{mixeq}) then (\ref{sigeq}) ($m < j$)} \\ &=c_m \sig{m}\del{m+1} & & \mbox{$j$ properly reduces $c_m$} \\ &=c_m & & \mbox{(\ref{mixeq})} 
\end{align*}

(c). By the cycle equations $$\dgn(c_m\del{m}) = \dgn(c_{m+1}\del{m}) = r-1$$ so $m$ reduces $c_m$. Apply part (a).

(d). In light of (a), we assume that $m$ does not reduce $c_m$. In light of (c), we assume that either $\dgn(c_{m+1}) >r$ (which will eventually lead to a contradiction) or that $m$ does not reduce $c_{m+1}$. By Lemma \ref{oprlem}, there are at least $r+1$ ordinals $0 \leq j \leq k-1$, $j \neq m$, that reduce $c_m$. By the assumptions just made, there are likewise at least $r+1$ ordinals $0 \leq j \leq k-1$, $j \neq m$, that reduce $c_{m+1}$. So if $k-1 < 2r+2$, then there is some $j$ that reduces both $c_m$ and $c_{m+1}$. Then
\begin{align*}  c_{m+1}\del{j} &=c_{j+1}\del{m+1} & & \mbox{cycle equations ($m < j$)} \\
&= c_m\sig{m}\del{j+1}\del{m+1} & & \mbox{Lemma \ref{techlem1} (b)} \\
&= c_m \sig{m} \del{m+1}\del{j} & & \mbox{simplicial identity}\\
&= c_m \del{j}
\end{align*}

Let $y_m\sig{j_1}\ldots\sig{j_r}$ be the Eilenberg-Zilber decomposition of $c_m$. Because $j$ reduces $c_m$, at least one of $\sig{j}$ or $\sig{j-1}$ appears in this sequence; let $s$ be the index such that $j_s=j$ if possible and $j_s=j-1$ otherwise. By repeated application of (\ref{mixeq}), the Eilenberg-Zilber decomposition for $c_m\del{j}$ is $$y_m\sig{j_1}\cdots \sig{j_{s-1}}\widehat{\sig{j_s}} \sig{j_{s+1}-1}\cdots\sig{j_r}.$$  Similarly, let $c_{m+1}= y_{m+1}\sig{i_1}\ldots\sig{i_{r'}}$ and let $i_t=j$ if possible and take $i_t=j-1$ otherwise. The Eilenberg-Zilber decomposition of $c_{m+1}\del{j}$ is $$y_{m+1}\sig{i_1}\cdots \sig{i_{t-1}}\widehat{\sig{i_t}} \sig{i_{t+1}-1}\cdots \sig{i_{r'}-1}$$ and by the above computation, these must be equal. Using the uniqueness statement, it follows that $r'=r$ (which means that $\dgn(c_{m+1})=r$), $s=t$, $y_m=y_{m+1}$, and the sequences of elementary degeneracies (excluding $\sig{j_s}$ and $\sig{i_t}$) agree. In particular, because $r>1$, we can apply (b) to conclude that $c_m=c_{m+1}$.
\end{proof}

Finally, we can prove Proposition \ref{techprop}. Unfortunately, despite the lengthy preparation, this will still be considerably harder than it was to prove the analogous result for the topos of cubical sets.

\begin{proof}[Proof of Proposition \ref{techprop}] We must show that $c_u = c_m \sig{m} \del{u}$ for all $0 \leq u \leq k$. Let $M$ be the set of indices which properly reduce $m$, and let $l$ be one greater than the largest element of $M$. As we saw in the proof of Lemma \ref{techlem2}, each $j \in M$ satisfies $m \leq j \leq k-2$ and $|M|=r$.

In Lemma \ref{techlem2}, we showed that $c_u = c_m \sig{m} \del{u}$ for the $r+2$ element set $$\{ m\} \cup \{ j+1 \mid j \in M\} \cup \{l+1\}.$$ In Lemma \ref{mplus1lem}, we proved the difficult case $u=m+1$. We divide the remaining proof into three parts, each with a few cases.

Part I: ($c_u = c_m \sig{m} \del{u}$ for all $u < m$). If $m=0$ this case is vacuous, so we may assume $m>0$. If $u<m$ then $c_{u}$ must be properly reduced by at least $r+1$ ordinals in $[k-2]$ since $c_{m}$ is of 
minimal degeneracy and the first ordinal of such degeneracy.  If $k-1 < 2r+2$ then $c_{u}$ must be properly reduced by at least one ordinal in the set $\{m-1\} \cup M$. Call this element $p$. Note that $u<m$ implies $u \leq p$ also. By the hypothesis, (\ref{cycleeq}), and (\ref{mixeq}), \begin{equation}\label{quickthing}c_u = c_u \del{p} \sig{p} = c_{p+1} \del{u} \sig{p} = c_{p+1} \sig{p+1} \del{u}.\end{equation}

If $p=m-1$, this is what we intended to show.

If $p=m$, then by part (c) of Lemma \ref{techlem1}, $c_m=c_{m+1}$ and by hypothesis $m$ properly reduces $c_m$. We use these facts to compute 
\begin{align*} 
c_u &= c_{m+1} \sig{m+1} \del{u} & & (\ref{quickthing}) \\ &= c_m \sig{m+1} \del{u} & & \mbox{Lemma \ref{techlem1} (c)} \\ &= c_m \del{m} \sig{m} \sig{m+1} \del{u} & & \mbox{$m$ properly reduces $c_m$} \\ &= c_m \del{m} \sig{m} \sig{m} \del{u} & & (\ref{sigeq}) \\ &= c_m \sig{m} \del{u} & & \mbox{$m$ properly reduces $c_m$} \end{align*} as desired. 

If $p>m$,\begin{align*}
c_{u} & =  c_{p+1}\sig{p+1}\del{u}  & &(\ref{quickthing}) \\
     & =  c_{m}\sig{m}\del{p+1}\sig{p+1}\del{u}  && \mbox{Lemma \ref{techlem1} (b)} \\   & =  c_{m}\del{p}\sig{p}\sig{m}\del{u} && \mbox{(\ref{mixeq}) then (\ref{sigeq}) ($m<p$)} \\
      & =  c_{m}\sig{m}\del{u}   && p\ \mbox{properly reduces}\ c_m \end{align*}
as desired.

Part II: ($c_u=c_m\sig{m}\del{u}$ for $u \geq m$ with $\dgn(u)>r$). By Lemmas \ref{techlem2} and \ref{mplus1lem}, $\dgn(u)>r$ implies that $u \not\in \{m,m+1 \} \cup \{j+1 \mid j \in M \} \cup \{l+1\}$. Let \[K= \{m\} \cup \{j+1 \mid j \in M, j+1<u \} \cup \{j \mid j \in M, j+1 > u \}.\] Because $M \subset \{m, \ldots, k-2\}$, $K \subset \{m, \ldots , k-1\}$. In fact, we can deduce that $k-1 \notin K$: $k-1 \in K$ is only possible if $u = k$ is properly reduced by $k-2$, in which case $u=l+1$, contradicting the above.

Because $u-1 \notin M$ and $u>m$, $|K|= r+1$.  Because $\dgn(c_u) > r$, there are at least $r+1$ elements of $[k-2]$ that properly reduce $c_u$. If $k-1 < 2r+2$, then there is some $p \in K$ that properly reduces $c_u$. We will use this $p$ to finish the proof for this case.

Case 1: ($p=m$).  We have \[ c_u = c_u \del{m} \sig{m} = c_m \del{u-1} \sig{m} = c_m \sig{m} \del{u} \] by (\ref{cycleeq}), (\ref{mixeq}), and the fact that we may take $u > m+1$. This is what we wanted. 

Case 2: ($m < p$, $u=p+1$). Inspecting the definition of $K$, we see that $p-1$ properly reduces $c_m$. By Lemma \ref{oprlem}, it follows that $p$ reduces $c_m$ so by Lemma \ref{techlem1}, $c_u = c_m\sig{m}\del{u}$ as desired.

Case 3: ($m < p$, $u>p+1$). As above, $u > p$ and $p \in K$ implies that $p-1 \in M$, which says that $p-1$ properly reduces $c_m$. We compute
\begin{align}
c_{u} & =  c_{u}\del{p}\sig{p} & & \makebox{$p$ properly reduces $c_{u}$}\notag \\
	& = c_{p}\del{u-1}\sig{p} & &  \makebox{cycle equations $(p<u)$} \notag \\
	& = c_{m}\sig{m}\del{p}\del{u-1}\sig{p} & & \makebox{Lemma \ref{techlem1} (b)} \notag \\
	& = c_{m}\sig{m}\del{p}\sig{p}\del{u} & & \makebox{simplicial identity ($p+1 < u$)} \label{thing} \\ \intertext{
If $p>m+1$, it follows that} 
c_{u} & =  c_{m}\del{p-1}\sig{p-1}\sig{m}\del{u} && \makebox{(\ref{mixeq}) then (\ref{sigeq})} \notag\\
	& =  c_{m}\sig{m}\del{u} && \makebox{$p-1$ properly reduces $c_{m}$} \notag \\ \intertext{as desired. If $p=m+1$, (\ref{thing}) simplifies to $c_u = c_m \sig{m+1} \del{u}$. We saw above that $p-1=m$ properly reduces $c_m$. It follows that} c_u &= c_m \sig{m+1} \del{u} \notag\\ &= c_m \del{m} \sig{m} \sig{m+1} \del{u} & & \mbox{$m$ properly reduces $c_m$}\notag \\ &= c_m \del{m} \sig{m} \sig{m} \del{u} & & \mbox{simplicial identity (\ref{sigeq})} \notag\\ &= c_m \sig{m} \del{u} && \mbox{$m$ properly reduces $c_m$}\notag\end{align} completing this case.

Case 4: ($m< u \leq p$). If $p \in K$ and $p \geq u$, then $p \in M$. This says that $p$ properly reduces both $c_u$ and $c_m$. We compute
\begin{align*}
c_{u} & = c_{u}\del{p}\sig{p} && \mbox{$p$ properly reduces $c_{u}$} \\
	& =  c_{p+1}\del{u}\sig{p} && \mbox{cycle equations ($u \leq p$)} \\
	& =  c_{m}\sig{m}\del{p+1}\del{u}\sig{p} && \mbox{Lemma \ref{techlem1} (b)} \\
	& =  c_{m}\del{p}\sig{p}\sig{m}\del{u} && \mbox{(\ref{mixeq}) then (\ref{mixeq}) then (\ref{sigeq})} \\
 	& =  c_{m}\sig{m}\del{u} && \mbox{$p$ properly reduces $c_{m}$} 
\end{align*}
 which is what we wanted to show.

Part III: ($c_u = c_m \sig{m} \del{u}$ for $u > m+1$ with $\dgn(u)=r$, not already covered by \ref{techlem2}). It remains to consider $u \not\in \{m,m+1 \} \cup \{i+1 \mid i \in M \} \cup \{l+1\}$. We use the set $K$ defined in Part II.

Because $\dgn(c_u) =r $, there are $r$ elements of $[k-2]$ that properly reduce $c_u$. Unless $c_u$ is properly reduced by the $r$ element set $\{k-1-r,\ldots, k-2\}$, Lemma \ref{oprlem} implies that there are $r+1$ elements of $[k-2]$ that reduce $c_u$. In this case, $k-1 < 2r+2$ implies that there is some $p \in K$ that reduces $c_u$, though not necessarily properly.

If $c_u$ is properly reduced by the set $\{k-1-r,\ldots, k-2\}$ and further if none of these ordinals lie in $K$, we must have $K \subset \{m,\ldots, k-2-r\}$, which necessitates $r \leq k-2-r$. We've assumed $k < 2r+3$, so the first inequality is an equality and $K = \{0,\ldots, r\}$, and hence $m=0$. The elements of $K$ all reduce $c_m$, so by Lemma \ref{techlem1} (b), we may assume $u > r+1$. It follows that $M = \{0,\ldots,m-1\}$ and $c_u$ is properly reduced by $r+1$, which we take for $p$ in this ``pathological'' case.

We will use the chosen $p$, however it was obtained, to complete the proof.

Case 1: ($p < u$). By the above, either $p \in K$ or $p$ is 2 greater than the maximal element of $M$. However we have chosen $p$, Lemma \ref{techlem2} applies.  Using the cycle equations, $$\dgn(c_u \del{p} ) = \dgn(c_p \del{u-1}) = r-1.$$ If $p =m$, this says that $u-1$ reduces $c_m$, so we're done by Lemma \ref{techlem1} (b). So we may suppose that $p > m$, in which case $u > m+1$. Then
\begin{align*} c_p\del{u-1} &= c_m \sig{m} \del{p}\del{u-1} & & \mbox{Lemma \ref{techlem2}} \\
&= c_m \del{u-1} \sig{m} \del{p} & & \mbox{(\ref{deleq}) then (\ref{mixeq}) $(m+1< u)$}
\end{align*}
The degeneracy of the left hand side is $r-1$ by the above calculation; hence, $u-1$ reduces $c_m$ by Lemma \ref{dgnlem}. We apply Lemma \ref{techlem1} (b) to achieve the desired result.

Case 2: ($p \geq u$). Note that $u > m+1$, so $p > m$ in this case. The inequality $p \geq u$ excludes the ``pathological'' case and implies that $p \in M$. Because $u > m+1$
\begin{align*}  c_{u}\del{p} &=c_{p+1}\del{u} & & \mbox{cycle equations} \\
&= c_m\sig{m}\del{p+1}\del{u} & & \mbox{Lemma \ref{techlem1} (b)} \\
&= c_m \del{u-1}\sig{m}\del{p} & & \mbox{(\ref{deleq}) then (\ref{mixeq}) ($m+1 < u$)}
\end{align*}
By hypothesis, $\dgn(c_u\del{p})= r-1$, so by Lemma \ref{dgnlem}, $u-1$ reduces $c_m$. Again apply Lemma \ref{techlem1} (b) to achieve the desired result.

Combining these (many) cases, we have shown that if $k<2r+3$ then $c_{u}=c_{m}\sig{m}\del{u}$ for all $u=0,1, \ldots ,k$. Hence $\sig{m}\del{u}$ is a filler for the $k$-sphere $c$ in $X$.
\end{proof}

In order to use Proposition \ref{techprop} to prove that $n$-skeletal implies $(2n-1)$-coskeletal, we must also show that the filler it constructs for high dimensional spheres is unique. This follows from the following lemma, which states that degenerate simplices are uniquely determined by their boundaries.

\begin{lem}\label{dgnbdrylem}
If $x$ and $y$ are degenerate simplices with the same faces, i.e., if $x\del{i}=y\del{i}$ for all $i$, then $x=y$.
\end{lem}
\begin{proof}
Since $x$ and $y$ are degenerate we can write them as $x=x'\sig{m}$ and $y=y'\sig{n}$. If $|m-n| \leq 1$ then without loss of generality $m \geq n$ and 
\[
x'=x'\sig{m}\del{m}=x\del{m}=y\del{m}=
y'\sig{n}\del{m}=y'
\] by (\ref{mixeq}). If $m=n$ it is clear that $x=y$. If $m = n+1$, \[ y' = y \del{n} = x \del{n} = x' \sig{n+1} \del{n} = x' \del{n} \sig{n} \] and \[ x= x' \sig{n+1} = x' \del{n} \sig{n} \sig{n+1} = x' \del{n} \sig{n} \sig{n} = y' \sig{n} = y\] by (\ref{sigeq}).

If $|m-n|>1$, then \begin{equation}\label{mneq}(\del{m}\sig{m})(\del{n}\sig{n}) = (\del{n}\sig{n})(\del{m}\sig{m})\end{equation} by the simplicial identities. By the computation in Example \ref{prex}, $m$ properly reduces $x$ and $n$ properly reduces $y$. In fact, using (\ref{mneq}), the same is true with $m$ and $n$ reversed: \[ x = x \del{m}\sig{m} = y \del{m}\sig{m} = y \del{n}\sig{n}\del{m}\sig{m} = y \del{m}\sig{m} \del{n} \sig{n} = x \del{m}\sig{m} \del{n} \sig{n} = x \del{n}\sig{n}\] and similarly $y = y \del{m}\sig{m}$. It follows that \[x = x \del{n} \sig{n} = y \del{n} \sig{n} = y.\quad\qedhere\]
\end{proof}

Any $0$-skeletal simplicial set is $1$-coskeletal: there exists a path of 1-simplices connecting each pair of vertices in any sphere of dimension $k >1$. It follows that each of the vertices are the same and the sphere can be filled by the unique degenerate $k$-simplex on that vertex.  Any $1$-skeletal simplicial set is $2$-coskeletal: any sphere of dimension $k>2$ contains at most one non-degenerate edge. It follows that the initial $s$ vertices are the same and the final $k+1-s$ vertices are the same. From this point, it is easy to identify the unique non-degenerate filler.

For larger $n$, we use the preceding work to prove our main result.

\begin{thm}\label{ssetupperboundthm}
If a simplicial set is $n$-skeletal with $n>1$, it is $(2n-1)$-coskeletal. Hence, the Aufhebung relation for the topos of simplicial sets is bounded above by $2n-1$. 
\end{thm}
\begin{proof} 
Let $X$ be an $n$-skeletal simplicial set and $c$ be a $k$-sphere in $X$ with $k>2n-1$. The case $n=2$ and $k=4$ can be proven by considering which degenerate 3-simplices have faces which satisfy the cycle equations. Such an argument does not require the difficult combinatorics of Proposition \ref{techprop}, and the details are left to the reader.

In general, the inequality $k>2n-1$ can be rewritten as \[ k < 2k-2n+1 = 2(k-1-n)+3.\] The faces of $c$ are $(k-1)$-simplices, which must have degeneracy at least $k-1-n$ which is greater than 1 in all cases which remain, so we may apply Proposition \ref{techprop} to conclude that $c$ has a filler. The filler is necessarily degenerate, so by Lemma \ref{dgnbdrylem} it's unique. This shows that $X$ is $(2n-1)$-coskeletal, as desired. 
\end{proof}

\begin{ex}\label{simpex} Let $X$ be the $n$-skeletal simplicial set, $n \geq 3$, generated by a single vertex $v$, distinct $(n-1)$-simplices $x'$ and $y'$ whose faces are degeneracies at $v$, and two $n$-simplices $x$ and $y$ with $x\del{0} = x'$, $y\del{n} = y'$, and all other faces of $x$ and $y$ degeneracies at $v$. Let $c$ be the simplicial $(2n-1)$-sphere with $$c_0 = \cdots = c_{n-1} = x\sig{0}\ldots \sig{n-3} \quad \text{and} \quad c_n = \cdots = c_{2n-1} = y\sig{n}\cdots \sig{2n-3}.$$ No simplex of $X$ contains both $x'$ and $y'$ as faces; hence, this sphere has no filler.
\end{ex}

\begin{thm}\label{ssetthm}
The Aufhebung relation for the topos of simplicial sets is $2n-1$.
\end{thm}
\begin{proof} Immediate from Theorem \ref{ssetupperboundthm} and the preceding example, which shows that an $n$-skeletal simplicial set is not necessarily $(2n-2)$-coskeletal.
\end{proof}

The results of this section can be used to compute a narrow bound on the Aufhebung relation for the topos of cyclic sets. We hope the details of this application will inspire others who are interested in comparing toposes which exhibit an analogous relationship.

Connes' cyclic category $\Lambda$ is a generalised Reedy category of interest to homotopy theorists \cite{connescohomologiecyclique}, \cite{dwyerhopkinskancyclic}, \cite{lodaycyclic}. It bears the following close relationship to $\Delta$: these categories have the same objects and a morphism $[n] \rightarrow [m]$ of $\Lambda$ can be written uniquely as a cyclic automorphism of $[n]$ followed by an arrow $[n] \rightarrow [m]$ of $\Delta$. Levels in the topos of \emph{cyclic sets}, that is, presheaves on $\Lambda$ again coincide with dimensions. However, restriction along the inclusion $\Delta \hookrightarrow \Lambda$ only respects the coskeletal inclusions of the essential subtoposes, which complicates the comparison. Nonetheless, the results of this section have the following corollary.

\begin{cor}\label{cycliccor} The Aufhebung relation for the topos of cyclic sets is between $2n-1$ and $2n+1$.
\end{cor}
\begin{proof}
The category $\Lambda$ is generated by the face and degeneracy maps of $\Delta$ together with cyclic automorphisms $\tau_n : [n] \rightarrow [n]$ of degree $n+1$ satisfying certain relations. See \cite[Ch.~6]{lodaycyclic} for details. The underlying simplicial set of a cyclic set is its image under the restriction functor $\Set^{\Lambda^{\op}} \rightarrow \Set^{\Delta^{\op}}$.

A cyclic set $X$ is $k$-coskeletal if and only if its underlying simplicial set is $k$-coskeletal: a $k$-sphere in a cyclic set $X$ is a morphism from the $(k-1)$-skeleton of the cyclic set represented by the object $[k] \in \Lambda$ to $X$. Concretely, such a sphere consists of the usual faces $c_0,\ldots, c_k$, together with rotations of these faces, satisfying certain relations. A simplicial sphere in a cyclic set determines a unique cyclic sphere of the same dimension: rotations of the faces will automatically satisfy the desired conditions. Furthermore, a filler for the simplicial sphere uniquely fills the cyclic sphere because the rotations of the simplicial filler will have the desired properties. Conversely, every filler for the cyclic sphere provides a filler for the underlying simplicial sphere in the underlying simplicial set. So a cyclic set is $k$-coskeletal as a cyclic set if and only if the underlying simplicial set is $k$-coskeletal.

By contrast, an $n$-skeletal cyclic set is  $(n+1)$-skeletal as a simplicial set. This follows most immediately from the presentation of the cyclic category $\Lambda$ as the category generated by the simplicial face and degeneracy maps together with an extra degeneracy map $\sig{n} : [n] \rightarrow [n-1]$ for each $n$. This ``extra degeneracy'' satisfies the analogous relations, except that $\sig{n}\del{0}$ is an automorphism of $[n-1]$ of order $n$; this was denoted $\tau_{n-1}$ above. An $n$-simplex in the image of $\sig{n}$ is degenerate, when $X$ is regarded as a cyclic set, but not when $X$ is regarded as a simplicial set. However, any epimorphism in $\Lambda$ can be expressed as a product $\sig{j_0}\cdots\sig{j_t}$ where an ``extra degeneracy'' appears as $\sig{j_0}$, if at all, and nowhere else. It follows that the dimension of a degenerate simplex changes at most by one when we regard the cyclic set as a simplicial set.

We may now compute a bound for the Aufhebung relation. Given an $n$-skeletal cyclic set, it is $(n+1)$-skeletal as a simplicial set, and so $(2n+1)$-coskeletal as a simplicial set, by Theorem \ref{ssetupperboundthm}. By the above discussion, this implies that the cyclic set is $(2n+1)$-coskeletal. Hence, the Aufhebung relation is at most $2n+1$.

For the lower bound, let $X$ be the cyclic set that is generated by the simplices described in Example \ref{simpex}. It is $n$-skeletal as a cyclic set. (Note however that its underlying simplicial set is $(n+1)$-skeletal and larger than the simplicial set described in the example.) The sphere described in the example cannot be filled for the reasons given above. So $X$ is an $n$-skeletal cyclic set which is not $(2n-2)$-coskeletal.
\end{proof}

We actually expect that the Aufhebung relation for cyclic sets is $2n-1$, based on the following intuition: the top dimensional non-degenerate simplices of an $n$-skeletal cyclic set, regarded as an $(n+1)$-skeletal simplicial set, are rotations of degenerate simplices, and we do not expect the process of rotation to substantially affect the combinatorics. We include Corollary \ref{cycliccor} more as an illustration of potential extensions of our results than as a definitive analysis of the essential subtoposes of this topos.

\bibliographystyle{alpha}
\bibliography{biblio}

\end{document}